\documentclass[11pt]{amsart}
\usepackage{amscd}
\usepackage{amsfonts}
\usepackage{color}
\usepackage{amsmath,amsthm,hyperref}
\usepackage[all]{xy}
\usepackage{amsmath,amsthm,hyperref}
\usepackage{amsmath,amssymb,amsthm,latexsym}
\usepackage{amscd}
\usepackage{multicol}
\usepackage{enumerate}
\newtheorem{theorem}{Theorem}[section]

\newtheorem{proposition}[theorem]{Proposition}

\newtheorem{definition}[theorem]{Definition}

\newtheorem{corollary}[theorem]{Corollary}

\newtheorem{lemma}[theorem]{Lemma}

\usepackage{amsmath}

\usepackage{amsfonts}

\usepackage{amsmath,amsthm}

\usepackage{amscd}

\usepackage[all]{xy}

\numberwithin{equation}{section}

\theoremstyle{remark}

\newtheorem{remark}[theorem]{Remark}
\newtheorem{example}[theorem]{\bf Example}
\newcommand{\R}{\mathbb{R}}
\newcommand{\C}{\mathbb{C}}
\newcommand{\D}{\mathbb{D}}

\newcommand{\dd}{\mathrm{d}}
\newcommand{\e}{\mathbf{e}}

\newcommand{\red}{\textcolor{red}}
\newcommand{\magenta}{\textcolor{magenta}}

\begin{document}

\title[Duality Theorem for Harmonic Maps]{\bf{A duality theorem for harmonic maps into inner symmetric spaces}}
\author{Josef F. Dorfmeister, Peng Wang }
\thanks{PW was partly supported by the Project 11971107 of NSFC. PW is thankful to the ERASMUS MUNDUS TANDEM Project for the financial support to visit the TU M\"{u}nchen. PW is also thankful to the Department of Mathematics of  the TU  M\"{u}nchen for its hospitality.
{The authors are thankful to the referee for many valuable suggestions which made the paper more readable, and in particular, for pointing out some incorrect formulations.}}

\begin{abstract}
	In this note, we show that for any  harmonic map into a non-compact inner symmetric space one can find naturally a ``dual" harmonic map into a compact inner symmetric space which can be constructed from the same basic data (called ``potentials" in the loop group formalism). Locally also the {converse}  duality theorem holds.
\end{abstract}

\date{\today}
\maketitle
\vspace{0.5mm}  {\bf \ \ ~~Keywords:}  Duality; Iwasawa decomposition; normalized potential; non-compact symmetric space; uniton.     \vspace{2mm}

{\bf\   ~~ MSC(2010): \hspace{2mm} 53A30, 53C30, 53C35}

%\tableofcontents

\section{Introduction}

The study of harmonic maps is an important topic in differential geometry which relates to many branches of mathematics. In the 1980s, it was known that harmonic maps from surfaces into compact symmetric spaces are  closely related with integrable systems.  An application of these techniques has lead to many new insights, see for instance \cite{Uh,DPW} and references therein.

Harmonic maps from surfaces into non-compact symmetric spaces are of  importance as well,  since such harmonic  maps appear naturally in many geometric and physical problems. But the study of harmonic maps into non-compact symmetric spaces is more complicated than the study of the compact case. In particular, globally smooth results are not so easy to come by, due to  technical difficulties, specific to the {non-compactness of the symmetric target space},  which cause the occurrence of singularities.

In this note, we derive a ``duality theorem" for harmonic maps into non-compact {inner} symmetric spaces. More precisely this means that for any  harmonic map into a non-compact inner symmetric space one can find naturally a harmonic map into a compact  {inner} symmetric space which can be constructed from the same basic data (called ``potentials" in the loop group formalism). Locally also the {converse} duality theorem holds.

Our  first  main result  relates a harmonic map into a non-compact inner symmetric space to a harmonic map  into a compact  {inner} symmetric space as follows:

\begin{theorem}\label{thm-noncompact}
	Let $G/K$ be an inner, non-compact symmetric space with $G$ being a real semisimple and  simply connected \footnote{  {We recall that the notion  ``simply connected"
includes the notion ``connected"}} Lie group. Let $\e$ be the identity element of $G$. Then there exists a maximal compact Lie subgroup $U$ of $G^{\C}$ such that $(  {U} \cap  {K}^{\mathbb{C}})^{\mathbb{C}}= {K}^{\mathbb{C}}$ and $ {U}/ ( {U} \cap  {K}^{\mathbb{C}} )$ is a compact, inner  symmetric space. Moreover, let $  M$  be a connected, simply connected Riemann surface with a basepoint $p_0 \in M$.
	\begin{enumerate}
		\item Let ${ {f}}:  {M}\rightarrow  {G}/ {K}$ be a harmonic map with    $f(p_0)=\e K$.
		Then there exists a  unique  harmonic map  ${{f}_ { {U}}}:  {M} \rightarrow   {U}/ ( {U} \cap  {K}^{\mathbb{C}})$
		into the compact, inner  symmetric space $ {U}/ ( {U} \cap  {K}^{\mathbb{C}} )$ which has the same normalized potential as ${ {f}}$ and $f_U(p_0)=\e ( {U} \cap  {K}^{\mathbb{C}} )$.

		\item Conversely,  let $ {f}_{ {U}}:  {M} \rightarrow {U}/ ( {U} \cap  {K}^{\mathbb{C}} )$ be a harmonic map with $f_U(p_0)=\e ( {U} \cap  {K}^{\mathbb{C}} )$.
		Then there exists a neighbourhood  $ {M}_0\subset  {M}$ of $p_0$ and a harmonic map
		$ {f} :  {M}_0 \rightarrow  {G}/ {K}$ which has the same normalized potential as
		$  {f}_{ {U}}$.
	\end{enumerate}
The harmonic maps $f$ and $f_U$ are related via the Iwasawa decompositions of their extended frames  w.r.t. different real loop groups $\Lambda G_{\sigma}$ and $\Lambda U_{\sigma}$ (see Section 2 for details).
\end{theorem}
Note that the difference in the above two directions is due to the fact that the Iwasawa decomposition for the loop groups of  compact Lie groups is  global,
 while the Iwasawa decomposition for loop groups of non-compact Lie groups is not global \cite{Ba-Do,B-R-S,PS}.  In \cite{Wang-a}, a Willmore two-sphere is presented which provides, together with its adjoint surface,  a harmonic map $f:S^2\setminus \gamma\rightarrow SO^+(1,7)/SO^+(1,1)\times SO(6)$  which is not defined on some circle $\gamma \subset S^2$ (see \cite{Wang-a}),  but for which the compact dual harmonic map  $f_{ {U}}:S^2\rightarrow SO(8)/SO(2)\times SO(6)$  is globally defined.

The {\it duality method}, as stated in the theorem above, applies to all harmonic maps from simply connected Riemann surfaces $M$ to symmetric spaces $G/K$; hence, in particular, to the case $M = S^2$.

If $M$ is not  {simply connected, then we denote by   $\hat M$  its} universal covering. The harmonic map
$f: M \rightarrow G/K$  can be lifted to a harmonic map $\hat f:\hat M\rightarrow G/K$. It is well-known that there exists an $S^1-$family of harmonic maps, named the associated family $\hat f_{\lambda}:\hat M\rightarrow G/K$, $\lambda\in S^1$, of $\hat f$ with $\hat f_{\lambda}|_{\lambda=1}=\hat f$ (\cite{DPW}, see also Section 2).
In general, $\hat f_{\lambda}$ will not descend to a  harmonic map defined on $M$ for all $\lambda\in S^1$.
 $f$ is called {\bf totally symmetric} if all harmonic maps $\hat{f}_{\lambda}$ descend to harmonic maps $f_{\lambda}:M\rightarrow G/K$. In other words,  the monodromy representation induced by $f_{\lambda}$ is trivial on $M$. It turns out that the duality theorem also holds for totally symmetric harmonic maps.
\begin{theorem}\label{totally symmetric} We retain the  {notation and the assumptions of
Theorem \ref{thm-noncompact}, except that now} $M$ is assumed to be simply connected. Then
$f$ is totally symmetric if and only if its compact dual surface $f_U$ is totally symmetric.
\end{theorem}

This result applies in particular to harmonic maps of  ``finite uniton type" (see Section 5),
i.e. to harmonic maps satisfying the assumptions (``totally symmetric") of Theorem \ref{totally symmetric} and admitting an extended frame which is a Laurent polynomial in the loop parameter $\lambda$.

The last condition is no restriction for compact $M$, since we prove
in Theorem \ref{M compact}  that all totally symmetric harmonic maps $f: M \rightarrow G/K$  defined on some compact Riemann surface $M$ are of finite uniton type. See Section 6 for illustrating examples.

As a consequence, there are many harmonic maps to which the results of this paper
do apply. In particular, our results are vital for the discussion of Willmore surfaces of finite uniton type  (see \cite{DoWa2} and Section 4 below), since the duality theorem permits to apply the work of Burstall-Guest \cite{BuGu}, to the investigation of Willmore surfaces in spheres.  Along this line, in \cite{Wang-iso}  a new  Willmore 2-sphere in $S^6$ has been produced which  solves a long open problem posed by Ejiri \cite{Ejiri1988}. Moreover, in \cite{Wang-1}, the normalized potentials of Willmore 2-spheres in $S^n$ are classified in the spirit of \cite{BuGu,DoWa2}.  More generally, for the (harmonic) conformal Gauss maps of Willmore surfaces in spheres $S^m, m \geq 3,$   where the target space is a non-compact symmetric space \cite{DoWa11,DoWa12},  substantial insights can be obtained by applying the duality theorem.   To be concrete, using this duality theorem, we can classify the conformal Gauss maps of Willmore 2-spheres by  classifying  all compact dual harmonic maps of these harmonic maps \cite{Wang-1}.

Note that compact, but not totally symmetric,   Willmore surfaces  {of genus $g > 0$} are,  in general, not of finite uniton type; but {to} the ones which are  of finite uniton type one can apply the above duality theorem (see Section 5).

To the authors' best knowledge,  the procedure outlined in this paper is so far the only {known method} to describe harmonic two-spheres, and moreover harmonic maps of finite uniton type, into non compact symmetric spaces. The main reason for this seems to be the following:  Morse theory  is a crucial tool for the investigation of harmonic two-spheres into compact symmetric spaces \cite{BuGu}, but that there is no Morse theory for the non-compact loop groups.

Finally we would like to remark that  we would expect the duality theorem also
to hold globally for certain harmonic maps $f:M \rightarrow G/K$, for which M is a non-simply-connected Riemann surface and where the monodromy representation induced by the harmonic map is non-trivial. This new class of harmonic maps shall be investigated in a separate publication.

We will recall the necessary results concerning harmonic maps in terms of loop groups in Section 2. In Section 3 we give a  proof of the {first} main result and in Section 4 we give a proof of the second main result. In Section 5 we discuss harmonic maps of finite uniton type and apply, in particular, the two main results to this important class of harmonic maps.
The paper ends with  some  concrete examples  illustrating our results presented in Section 6.\footnote{ Following the suggestion of  {an} anonymous referee we have divided the paper entitled ``Willmore surfaces in spheres via loop groups I: generic cases and some examples "(arXiv:1301.2756)
into three parts. The present paper is part III  and \cite{DoWa11,DoWa12} are part I and part II respectively. }

%%%%%%%%%%%%%%%%%%%%%%%%%%%%

\section{The DPW method for harmonic maps}

\subsection{Harmonic maps into  symmetric spaces} \label{sect2.1}

Let $G$ be a simply-connected, semi-simple, real Lie group, with $\e\in G$ be its identity element.
Let $G/K$ be a symmetric space defined by an involution $\sigma: G\rightarrow G$ such that $(Fix^\sigma(G))^0\subset K\subset Fix^\sigma(G)$. Here
$Fix^\sigma(G)=\{g\in G|\sigma(g)=g\}$ and $(Fix^\sigma(G))^0$ denotes the connected component of $Fix^\sigma(G)$ containing $\e$.

 Let $\pi:G\rightarrow G/K$ be the projection from $G$ to $G/K$. Let $\mathfrak{g}$ and $\mathfrak{k}$ be the Lie algebras of $G$ and $K$ respectively. The generalized Cartan decomposition $\mathfrak{g}=\mathfrak{k}\oplus\mathfrak{p}$ (relative to $\sigma$) satisfies
$[\mathfrak{k},\mathfrak{k}]\subset\mathfrak{k},$
$[\mathfrak{k},
\mathfrak{p}]\subset\mathfrak{p},
\ [\mathfrak{p},\mathfrak{p}]\subset\mathfrak{k}.$
Let $f:M\rightarrow G/K$ be a  harmonic map from a connected  Riemann surface $M$ to the symmetric space $G/K$. Let $\hat{\pi}_M :\hat{M} \rightarrow M$ denote the universal covering of $M$. Then $\hat{f} = f \circ \hat{\pi}_M : \hat{M} \rightarrow G/K$
is harmonic and  there exists a frame $F:  \hat{M} \rightarrow G$ of
$\hat{f} = f \circ \hat{\pi}_M$.  The frame $F$ is called ``frame of $f$" as well as ``frame of $\hat{f}$".
Then for the Maurer-Cartan form and the Maurer-Cartan equation we obtain:
\[
F^{-1}\dd F= \alpha, \hspace{4mm} \mbox{and} \hspace{4mm } \dd \alpha+\frac{1}{2}[\alpha\wedge\alpha]=0.
\]
Decomposing these equations w.r.t. $\mathfrak{g}=\mathfrak{k}\oplus\mathfrak{p}$, we have $\alpha=\alpha_{\mathfrak{k}}+\alpha_{\mathfrak{p}}$ with $\alpha_\mathfrak{k}\in \Gamma(\mathfrak{k}\otimes T^*M)$ and $\alpha_{\mathfrak{k}}\in \Gamma(\mathfrak{p}\otimes T^*M)$.
Decompose $\alpha_1$ into the $(1,0)-$part $\alpha_\mathfrak{p}^\prime$ and the $(0,1)-$part $\alpha_\mathfrak{p}^{\prime \prime}$. Set
 \[
\alpha_{\lambda}=\lambda^{-1}\alpha_{\mathfrak{p}}^\prime+\alpha_\mathfrak{k}+\lambda\alpha_{\mathfrak{p}}^{\prime \prime},  \quad \lambda\in S^1.
\]
\begin{lemma} $($\cite{DPW}$)$ The map  $f:M\rightarrow G/K$ is harmonic if and only if
\[ \dd \alpha_{\lambda}+\frac{1}{2}[\alpha_{\lambda}\wedge\alpha_{\lambda}]=0\ \ \hbox{for all}\ \lambda \in S^1.
\]
\end{lemma}
\begin{definition} Assume $\hat M$ has a complex coordinate $z$ with $0\in\hat M$, and assume w.l.g. that $F(0,0)=\e$. Let $F(z,\bar{z},\lambda)$ be the solution
on $\hat{M}$ to the equation
\[\dd F(z,\bar{z},\lambda)= F(z,\bar{z}, \lambda) \, \alpha_{\lambda},\  F(0,0,\lambda)=\e.\]
For simplicity, we will  sometimes write  $F_{\lambda}(z, \bar z)$  or simply $F_{\lambda}$ for $F(z,\bar{z}, \lambda)$. Also  note that $F(z,,\bar{z},\lambda)|_{\lambda=1}=F(z,\bar{z})$ holds.
Moreover,  for each $\lambda \in S^1$ the frame $F_\lambda$ is the {\em extended frame}  of a harmonic map $\hat{f}_\lambda : \hat{M} \rightarrow G/K$. The family of harmonic maps $\hat{f}_\lambda$ will be called the {\em associated family} associated to $\hat{f}$ or simply ``associated to $f$".
 \end{definition}

 \begin{remark}   The family of maps $\hat{f}_\lambda$ actually makes sense for all $\lambda \in \C^*$, but will generally only be harmonic for $\lambda \in S^1$.

 In general, the map $\hat{f}_\lambda :\hat{M} \rightarrow G/K$ can not be pushed down to become a harmonic map  from $M$ to $G/K$. So the associated family is generally only defined on $\hat{M}.$
 Of course, by construction, the harmonic map  $\hat{f}_{\lambda = 1}$ is the canonical lift of $f$ and thus descends to the originally given  harmonic map $f:M \rightarrow G/K$.
 \end{remark}
%%%%%%%%%%%%%%%%%%%%%%%%%

\subsection{The DPW construction of harmonic maps}

\subsubsection{Two decomposition theorems}

Let $G$ be a  simply connected  real Lie group and $G^\C$ its  {simply connected} complexification
(For details on complexifications, in particular of semi-simple Lie groups, see \cite{Hochschild}.
 In particular, we will always assume that $G^\C$ is realized as a group of linear transformations.)
Let $\sigma$ denote an involution of $G$ and $K$ a closed subgroup of  {$G$} satisfying $(Fix^\sigma(G))^0 \subset K \subset Fix^\sigma(G)$.
Then $\sigma$ fixes $\mathfrak{k} = Lie K$.
The extension of $\sigma$  to an involution of $G^\C$ has
 $\mathfrak{k}^\C$  as  its fixed point algebra.
By abuse of notation we put $K^\C = (Fix^\sigma(G^\C))^0$.

Let $\Lambda G^{\mathbb{C}}_{\sigma}$ denote the group of loops in $G^\C$ twisted by $\sigma$. Let $\Lambda^+G^{\mathbb{C}}_{\sigma}$  denote the subgroup of loops which extend holomorphically to the  open unit disk $|\lambda| <1$. Set
\[
\Lambda_B^+ G^{\mathbb{C}}_{\sigma}:=\{\gamma\in\Lambda^+G^{\mathbb{C}}_{\sigma}~|~\gamma|_{\lambda=0}\in \mathfrak{B} \}.
\]
Here $\mathfrak{B}\subset K^{\mathbb{C}}$ is defined from the {classical unique } Iwasawa decomposition
$K^{\mathbb{C}}=K\cdot\mathfrak{B}.$
Let $\Lambda^-_*G^{\mathbb{C}}_{\sigma}$ denote the loops that extend holomorphically into $\infty$ and take the value $\e$ at infinity.
The following two results usually are stated as part of a more detailed description of group decompositions,
namely   ``Birkhoff decomposition" and ``Iwasawa decomposition"  respectively.  But for this paper the stated results and the ones of the following theorem suffice.
 {Note that  $\Lambda G_{\sigma}\times \Lambda^+_{B}G^{\mathbb{C}}$ acts naturally on
$\Lambda G^{\mathbb{C}}_{\sigma}$ by $(A,V_+).g = AgV_+^{-1}$. The corresponding cosets are called ``Iwasawa cells".}

\begin{theorem} \label{thm-iwasawa}\
\begin{enumerate}
\item The multiplication $\Lambda^-_* G^{\mathbb{C}}_{\sigma}\times \Lambda^+G^{\mathbb{C}}\rightarrow\Lambda G^{\mathbb{C}}_{\sigma}$ is a  biholomorphic map onto the open and dense subset $ \Lambda^-_* G^{\mathbb{C}}_{\sigma}\cdot \Lambda^+G^{\mathbb{C}}$ (big Birkhoff cell) of $\Lambda G^{\mathbb{C}}_{\sigma}$.
\item The multiplication $\Lambda G_{\sigma}\times \Lambda^+_{B}G^{\mathbb{C}}\rightarrow\Lambda G^{\mathbb{C}}_{\sigma}$ is a real analytic diffeomorphism onto the open  subset $\Lambda G_{\sigma}\cdot \Lambda^+_{B}G^{\mathbb{C}} \subset \Lambda G^{\mathbb{C}}_{\sigma}$.\\
     Moreover, if the group $G$ is compact, then the last two sets are equal.
     \item {The union of all open Iwasawa cells is dense in $ \Lambda G^{\mathbb{C}}_{\sigma}$.}
\end{enumerate}
\end{theorem}
\begin{proof}
{For the convenience of the reader we give a proof.  It is easy to verify that all the groups
occurring are Banach Lie groups and, more precisely, Banach Lie subgroups of
$\Lambda G^{\mathbb{C}}_{\sigma}$. Moreover, the Lie algebras of the
Banach Lie groups on the left side of the conclusion add up to the Lie algebra
$Lie \Lambda G^{\mathbb{C}}_{\sigma}$  of the right side and even the corresponding groups have trivial intersection.
In view of   Corollary 3.1.4. of \cite{DGS}  {the first parts of $(1)$ and $(2)$ are proven. For the second part of $(1)$  it now only remains to prove that the image set in question } is dense in $\Lambda G^{\mathbb{C}}_{\sigma}$. But this follows from the properties of the holomorphic section discussed in \cite{PS}, see Corollary 2.4 and Corollary 2.5 of \cite{DPW} and observe that the argument given there applies to  our norm as well.

The  {second  statement of $(2)$} is well-known (see e.g. \cite{PS}).

Finally, $(3)$ follows from Proposition 5.4 in \cite{Do}.}
\end{proof}

%%%%%%%%%%%%%%%%%%%%%%%%%%%
\subsubsection{The DPW construction }

Here we recall the DPW construction for harmonic maps. Let $\mathbb{D}\subset\mathbb{C}$ be a disk or $\mathbb{C}$ itself, with complex coordinate $z$.

\begin{theorem} \label{thm-DPW}\cite{DPW}
We retain the notation and the assumptions introduced above.
\begin{enumerate}
\item Let $f:\mathbb{D}\rightarrow G/K$ be a harmonic map with an extended frame $F(z,\bar{z},\lambda) \in \Lambda G_{\sigma}$ based at $z=0$, i.e., $F(0,0,\lambda)=\e$. Then there exists a Birkhoff decomposition
\[
F_-(z,\lambda)=F(z,\bar{z},\lambda) F_+(z,\bar{z},\lambda),~ \hbox{ with }~ F_+(z, \bar z, \lambda) \in\Lambda^+G^{\mathbb{C}}_{\sigma}, \ F(z, \bar z, \lambda) \in  \Lambda G_{\sigma},
\]
such that $F_-(z,\lambda):\mathbb{D} \rightarrow\Lambda^-_*G^{\mathbb{C}}_{\sigma}$ is meromorphic  {on $\D$.} Moreover, the Maurer-Cartan form of $F_-$ is of the form
\[
\eta=F_-^{-1}\dd F_-=\lambda^{-1}\eta_{-1}(z)\dd z,
\]
with $\eta_{-1}$ independent of $\lambda$. The $1$-form $\eta$ is called the {\emph normalized potential} of $f$.
\item Let $\eta$ be a $\lambda^{-1}\cdot\mathfrak{p}-$valued  meromorphic 1-form on $\mathbb{D}$
 {which is holomorphic at $z =0$.} Let $F_-(z,\lambda)$ be a  {meromorphic} solution to $F_-^{-1}\dd F_-=\eta$, $F_-(0,\lambda)=\e$  {which is defined globally  on $\D$.}
Then, on an  open  subset $\mathbb{D}_{\mathfrak{I}}$ of $\mathbb{D}$ containing the base
point $z =0$,
\[
 {F_-(z,\lambda)}=\tilde{F}(z,\bar{z},\lambda)\cdot { \tilde{F}_+(z,\bar{z},\lambda)},\ \hbox{ with }\ \tilde{F}\in\Lambda G_{\sigma},\  {\hbox{ and }  \tilde{F}_+}\in\Lambda ^+_{B} G^{\mathbb{C}}_{\sigma}
\]
{real analytic}. This way, one obtains an extended frame $\tilde{F}(z,\bar{z},\lambda)$ of some harmonic map from
$\mathbb{D}_{\mathfrak{I}}$  to $G/K$ with $\tilde{F}(0,0,\lambda)=\e$. In particular, if $G$ is compact, then  one can choose $\mathbb{D}_{\mathfrak{I}} =\D \setminus \{ \mbox{poles of} \hspace{2mm} \eta  \}.$
Moreover, all harmonic maps can be obtained in this way, and these two procedures are inverse to each other if the normalization $\e$ at some base point is used.
\end{enumerate}
\end{theorem}
For the case of two-spheres we refer to Theorem 3.11 of \cite{DoWa12}.

%%%%%%%%%%%%%%%%%%%%%%%%%%

\section{  Proofs of the  Duality Theorems}

In this section we  consider a connected,  simply connected, non-compact, semi-simple, real  Lie group $G$
and a non-compact, inner, pseudo-Riemannian symmetric space $G/K$ defined by $\sigma$.
We will first  {construct the } dual, compact symmetric space $ {U}/( {U}\cap  {K}^{\mathbb{C}})$  of $G/K$,
and will show that it is also inner. Then we will derive proofs of Theorem 1.1 and 1.2.

\subsection{The dual inner, compact symmetric space }

Set $\mathfrak{g} ^\C=\mathfrak{g}\otimes \C$ for $\mathfrak{g} =Lie(G)$.
Denote by $\tau$ the complex anti-linear involution defining $\mathfrak{g}$ in  $\mathfrak{g} ^\C$. Obviously $\sigma\tau=\tau\sigma$.
Let $\theta $ be some Cartan involution of $\mathfrak{g} ^\C$ commuting with $\sigma$ and $\tau$.
Let  $ {G}^\C$ denote the simply connected complexification of $ {G}$ \cite{Hochschild}, with Lie algebra $\mathfrak{g} ^\C$.  {Then}  $\sigma, \tau$ and $\theta$ have extensions to pairwise commuting involutive group homomorphisms of  $ {G}^\C$.
Since $G$ is  {a semi-simple Lie group}, we can assume that the natural image of $ {G}$ in $ {G}^\C$ is a closed subgroup of $ {G}^\C$ \cite{Hochschild}.
Let $ {U} = Fix ^{\theta}( {G}^\C)$.
Then $ {U}$ is a maximal compact subgroup of  $ {G}^\mathbb{C}$, and $ {U}$ is  { simply connected}
( See for example, page 3, Lemma 2 of \cite{Aom}).
 Moreover, observing that
$ {K}^\mathbb{C} = (Fix ^{\sigma }( {G}^\C))^0 \subset  {G}^\C$ is a  connected complex Lie group    satisfying
$ {K}^\C \cap  {G} =  {K}$, we have

\begin{definition}
The symmetric space $ {U}/( {U}\cap  {K}^{\mathbb{C}})$ is called the dual (compact) symmetric space of $G/K$.
\end{definition}

\begin{lemma}\label{lemma-inner}  { We retain the notation and the assumptions above.} If $G/K$ is an inner symmetric space, then
the symmetric space $ {U}/( {U}\cap  {K}^{\mathbb{C}})$ is also an inner symmetric space.
\end{lemma}

\begin{proof}
Let $\mathfrak{g} = \mathfrak{k} + \mathfrak{p}$
be the decomposition of $\mathfrak{g}$ relative to $\sigma$ and $\mathfrak{g} = \mathfrak{h} + \mathfrak{m}$
the decomposition of $\mathfrak{g}$ relative to $\theta$.
Then
\[\mathfrak{g} = \mathfrak{k} \cap  \mathfrak{h} + \mathfrak{k} \cap \mathfrak{m} +
\mathfrak{p} \cap\mathfrak{h} + \mathfrak{p} \cap \mathfrak{m}\]
as a direct sum of vector spaces. Moreover, for the Lie algebra $\mathfrak{u}$ of $ {U}$   we have
\begin{equation*} {\begin{split}\mathfrak{u} & = \mathfrak{k} \cap  \mathfrak{h} + \mathfrak{p} \cap \mathfrak{h} +
i \left( \mathfrak{k} \cap  \mathfrak{m} + \mathfrak{p} \cap \mathfrak{m} \right)= \left(\mathfrak{k} \cap  \mathfrak{h} + (i \mathfrak{k}) \cap (i\mathfrak{m}) \right) +
\left(\mathfrak{p} \cap  \mathfrak{h} + (i \mathfrak{p}) \cap (i\mathfrak{m}) \right)\\
&= \mathfrak{k^{\mathbb{C}}} \cap \mathfrak{u}  + \mathfrak{p^{\mathbb{C}}} \cap   \mathfrak{u} .\end{split}}
\end{equation*}
It is easy to see now that $ \left( \mathfrak{k}^{\mathbb{C}} \cap \mathfrak{u}\right)^{\mathbb{C}} = \left(\mathfrak{k} \cap  \mathfrak{h} + (i\mathfrak{k}) \cap(i \mathfrak{m}) \right)^{\mathbb{C}} = \mathfrak{k}^{\mathbb{C}} $ holds.
\begin{equation} \label{about-k}
\left( \mathfrak{k}^{\mathbb{C}} \cap \mathfrak{u} \right)^{\mathbb{C}} = \mathfrak{k}^{\mathbb{C}}
\end{equation}
As a consequence, for the maximal compact Lie subgroup $ {U}$ of $ {G}^{\mathbb{C}}$ constructed above we obtain
\begin{equation} \label{innerU}
( {U}\cap  {K}^{\mathbb{C}})^{\mathbb{C}}= {K}^{\mathbb{C}}
\end{equation}
 which follows, since both sides represent connected Lie subgroups of $ {G}^{\mathbb{C}}$
  by the  Springer-Steinberg Theorem (see  \cite{Helgason}, chapter VII, Theorem 8.2),
  and  have the same Lie algebra.
Since $\sigma$ is inner, there exists a maximal abelian subalgebra $\mathfrak{a}$ of $\mathfrak{k}$ which already is maximal abelian in $\mathfrak{g}$
(see, e.g., \cite{Helgason}, chapter  IX, Theorem 5.7)). Using  { (\ref{about-k}) and  (\ref{innerU})}  it then follows  $rank (\mathfrak{u}\cap \mathfrak{k}^{\C})= rank(\mathfrak{k})=rank (\mathfrak{g})=rank(\mathfrak{u})$. Now Theorem 5.6 of chapter IX, \cite{Helgason}, implies that the restriction of $\sigma$ to $ {U}$ is inner. This finishes the proof of the lemma.
\end{proof}

\subsection{Duality theorems}

We first assume that $M$ is simply connected as in Theorem \ref{thm-noncompact}.
Let ${ {f} :  {M} }\rightarrow G/K$ be a harmonic map.
 Then we also have an extended frame
$ {F}: {M}  \rightarrow \Lambda  {G}_{\sigma}
\subset\Lambda  {G}^{\mathbb{C}}_{\sigma}$.
(Note that here for $G = SL(2,\R)$ the last two ``inclusions''  actually describe inclusions of the image of $ {G}$ under the natural homomorphism.  Also recall that by \cite{DoWa12} in the case $M = S^2$ the frame $F$ is permitted/required  to have singular points.)

\begin{example}
For a strongly conformally harmonic map $f:M\rightarrow SO^+(1,2m+3)/SO(1,3)\times SO(2m)$ associated with a strong Willmore map $f:  {M}\rightarrow S^{2m+2}$ \cite{DoWa11}, we have $ {G} \cong Spin(1,2m+3)^0$,
 $ G^{\mathbb{C}} \cong  SO(1, 2m+3, \C)$ and $  {G}^\C \cong Spin(1,2m+3,\mathbb{C})$.
Moreover, we have
$ K=Spin(1,3)\times Spin(2m)$ and $  {K}^\C = Spin (1,3,\C) \times Spin (2m,\C).$
Note that $G$ is the (connected) simply connected covering of $SO^+(1,2m+3)$.
 Hence $ {U}\cong Spin(2m+4),$ and
$ {U} \cap  {K}^{\mathbb{C}}=Spin(4)\times Spin(2m)$. On the Lie algebra level, we have
 \[\begin{split}&Lie( {U})=\mathfrak{u}=\{A\in \mathfrak{so}(1,2m+3,\mathbb{C})|A=(a_{jk}),\ ia_{1j}\in\mathbb{R},j=1,\cdots,2m+4\},\\
  &(\mathfrak{u}\cap\mathfrak{k}^{\mathbb{C}})^{\mathbb{C}}=\mathfrak{so}(1,3,\mathbb{C})\times \mathfrak{so}(2m,\mathbb{C})=\mathfrak{k}^{\mathbb{C}}.
\end{split}\]
\end{example}

 \ \\
{\em Proof of Theorem \ref{thm-noncompact}.}
 Restricting $\sigma$ to $ {U}$, we  consider the twisted loop group $\Lambda  {U}^{\mathbb{C}}_{\sigma}$ and the
corresponding Iwasawa decomposition of $\Lambda  {U}^{\mathbb{C}}_{\sigma}$ relative to $\Lambda  {U}_{\sigma}.$
The first key observation is that $\Lambda  {G}^{\mathbb{C}}_{\sigma} =\Lambda  {U}^{\mathbb{C}}_{\sigma}$
 {holds}.
In fact, by construction, complexifying $ {G}$ and complexifying $ {U}$ yields the same complex Lie group $ {G}^{\mathbb{C}}$
and the holomorphic extension of $\sigma$, considered as an involution of $ {G}$ or considered as an involution of  $ {U},$ yields the same involution of $ {G}^{\mathbb{C}}$.
Therefore the  complex twisted loop group, constructed by starting from $ {G}$ or starting from $ {U}$ are the same, that is,
\[\Lambda  {G}^{\mathbb{C}}_{\sigma} =\Lambda  {U}^{\mathbb{C}}_{\sigma},\ \hbox{ and } \ \Lambda^{+}  {G}^{\mathbb{C}}_{\sigma}=\Lambda^{+} {U}^{\mathbb{C}}_{\sigma}.\]
Applying this  {and the fact $U$ being compact,} we can also perform the following  {(global)} Iwasawa decomposition of $\Lambda  {G}^{\C}_{\sigma}$:
\begin{equation} \label{otherIwasawa}
\Lambda  {U}_{\sigma} \cdot\Lambda^{+}  {G}^{\mathbb{C}}_{\sigma}=\Lambda  {G}^{\mathbb{C}}_{\sigma}.
\end{equation}
Now let us turn to the harmonic maps. First we assume that $ {M}=\mathbb{D}$ is a contractible open subset
 of $\C$.  {Then we obtain a global extended frame $F(z,\bar{z},\lambda)$ of
$ { {f}}$ on $\D$.}
 Note, here and in what follows we will use  for simplicity  $F$ for $F(z,\bar{z},\lambda)$. Applying the decomposition (\ref{otherIwasawa}) to the frame $F$ we obtain
\begin{equation}\label{eq-noncompact}
F=F_{ {U}}\cdot W_{+},\hspace{3mm} F_{ {U}}\in \Lambda  {U}_{\sigma},\hspace{3mm}  W_{+}\in \Lambda^{+} {U}^{\mathbb{C}}_{\sigma}.
\end{equation}
Writing as usual  $\alpha=F^{-1}\dd F=\lambda^{-1}\alpha_{\mathfrak{p}}'+
\alpha_0+\lambda\alpha_{\mathfrak{p}}'',$
we obtain
\[F_{ {U}}^{-1}\dd F_{ {U}}=\alpha_{ {U}}=W_{+}\alpha W_{+}^{-1}-\dd W_{+}W_{+}^{-1}=
\lambda^{-1}W_0\alpha_{\mathfrak{p}}'W_0^{-1}+
\alpha_{ {U},0} + \lambda \alpha_{ {U},1}+ O(\lambda^2).\]
Since $W_+\in \Lambda^+  {U}^{\mathbb{C}}_{\sigma}=\Lambda^+  {G}^{\mathbb{C}}_{\sigma}$, we have
$W_0\in  {K}^{\mathbb{C}},$  and $ \sigma( W_0 )= W_0.$
Since $\alpha_{\mathfrak{p}}'\in   \mathfrak{p}^{\mathbb{C}}$,
we obtain moreover $$\sigma(\alpha_{\mathfrak{p}}')=-\alpha_{\mathfrak{p}}', ~~ \hbox{ and }~~ \sigma (W_0\alpha_{\mathfrak{p}}'W_0^{-1})=-W_0\alpha_{\mathfrak{p}}'W_0^{-1}.$$
Since $\alpha_{ {U}}$ is fixed by the anti-holomorphic involution $\theta$ we infer
\[\alpha_{ {U}}=\lambda^{-1} {\alpha}'_{\mathfrak{p}}+
 \alpha_{\mathfrak{k}}+
\lambda {\alpha}''_{\mathfrak{p}},~~\hbox{ with }~~ \alpha_{\mathfrak{k}} \in \mathfrak{u} \cap \mathfrak{k}^{\mathbb{C}},\hbox{ and  }
 {\alpha}''_{\mathfrak{p}} = \theta \left(  {\alpha}'_{\mathfrak{p}}\right) \in
\mathfrak{p}^{\mathbb{C}}.\]
As a consequence, $F_{ {U}}$ is the frame of a harmonic map ${ {f}_{  U} }:  {M} \rightarrow  {U}/( {U}\cap  {K}^{\mathbb{C}})$, where actually
\[{ {f}_{ {U}}}=F_{ {U}} \mod \  {U} \cap { K}^\mathbb{C}.\]
Computing the Birkhoff decomposition of $F$ as well as the Birkhoff decomposition of $F_{ {U}}$ we obtain
\[F_-F_+=F=F_{ {U}}W_+=F_{ {U},-}\cdot F_{ {U},+}\cdot W_+,\]
with
\[F_-=I+O(\lambda^{-1}), \hspace{3mm} F_{ {U},-}=
I+O(\lambda^{-1})\in \Lambda^-_* {U}^{\mathbb{C}}_{\sigma}= \Lambda^-_* {G}^{\mathbb{C}}_{\sigma}.\]
This implies
$F_-=F_{ {U},-},$ whence we also have  $\eta=F_-^{-1}\dd F_-=F_{ {U},-}^{-1}\dd F_{ {U},-}.$

 Now it suffices to discuss the case $M = S^2$.
 Following the argument of the proof of Theorem 3.11 in \cite{DoWa12},
we remove two points $p_1,p_2$ from $S^2$,  {w.l.g.,} the two poles, and obtain two frames $F_1$ and $F_2$ defined on $U_1 = S^2 \setminus \{p_1\}$ and  $U_2 = S^2 \setminus \{p_2\}$ respectively.
Since we also assume that these two frames both take the value $\e$ at the  basepoint $p_0\in S^2 \setminus \{ p_1,p_2\} $ we obtain
$F_1 = F_2 k$ on $S^2 \setminus \{ p_1,p_2\}$,
where  $k \in K$. Following the argument above we form the Iwasawa decompositions
relative to the maximal compact group $U$ and obtain with obvious notation
$F_{1,U} W_{1,U+} = F_{2,U} W_{2,U,+} k$. Combining the three matrices in $\Lambda^+G^\C_\sigma$ we thus obtain $F_{1,U} = F_{2,U} W_{U,+} .$ But this implies $W_{U,+} \in U \cap K^\C.$
As a consequence we obtain for the harmonic maps  {$f_{U,j}, j = 1,2$ associated to} $F_{j,U} $ on $U_j,$ the relation
$ f_{U,1} = f_{U,2}$ on $U_1 \cap U_2$. But, by construction, $f_{U,1}$ extends smoothly through $p_2$ and
$f_{U,2}$ extends smoothly through $p_1,$ whence the two harmonic maps combine to yield a harmonic map on all of $S^2$.

Conversely, let ${ {f}_{ {U}}}: M\rightarrow {U}/ ( {U} \cap  {K}^{\mathbb{C}} )$ be a harmonic map
with ${ {f}_{ {U}}}|_{p_0}=\e ( {U} \cap  {K}^{\mathbb{C}} )$ and let $F_{ {U}}\in \Lambda U_\sigma$  be an extended frame for
${ {f}_{ {U}}}$ satisfying $F_{ {U}}|_{p_0}=\e$.
Then there exists a neighbourhood  $M_0\subset M$ of $0$ on which
 $F_{ {U}} = F V_+$ holds with $F\in \Lambda G_\sigma$ and $V_+\in \Lambda^{+} {U}^{\mathbb{C}}_{\sigma}.$ Then ${ {f}} \equiv F \mod K $ satisfies the claim. \hfill$\Box $

\begin{remark} \label{duality}
  We would like to point out that the last part of  Theorem \ref{thm-noncompact} shows that the
``duality" relating  harmonic maps into compact symmetric spaces to  the ones in the dual non-compact symmetric spaces is in general only local, due to the fact that the corresponding Iwasawa decompositions for non-compact symmetric spaces is in general not global. It would be interesting to understand this duality in a more global sense.

\end{remark}

\begin{remark} An anonymous referee of this paper has asked about the dependence of the loop group
method under a change of base point. We appreciate this question very much, but would like to
refer the interested reader for a detailed discussion of this issue to \cite{Do-base}.
\end{remark}
%%%%%%%%%%%%%%%%%%%%%%%%%%%%%%

\section{Totally symmetric harmonic maps}

%%%%%%%%%%%%%%%%%%%%%%%%%%%%%%%%

\subsection{Full harmonic maps}

\begin{definition}
We retain the notation introduced in Section \ref{sect2.1}.
The harmonic map $f : M \rightarrow G/K$ is called {\bf full} if and only if  each $g \in G$
satisfying $gf(p) = f(p)$  for all $p \in M$ is contained in the center  of $G$, where
\begin{equation}
Center(G) = \{ g \in G; gh = hg \hspace{2mm} \mbox{for all} \hspace{2mm} h \in G \}.
\end{equation}
\end{definition}
We will need the following
\begin{proposition}
Assume the harmonic map $f : M \rightarrow G/K$ is full, then also each harmonic map in the associated family of $f$ is full.
\end{proposition}
\begin{proof}
 Assume the harmonic map $f : M \rightarrow G/K$ is full.
 We can assume w.l.g. that $M$ is simply connected. Let $f_\lambda$ denote the associated family of $f$.
 Then $f_\lambda \equiv F_\lambda \mod K$, where $F_\lambda(0,0) = I, \ F_\lambda^{-1} \dd F_\lambda = \alpha_\lambda,$ and $\alpha_\lambda = \lambda^{-1} \alpha'_\mathfrak{p} +  \alpha_\mathfrak{k} +  \lambda \alpha^{\prime \prime}_\mathfrak{p}, $ with $\alpha_{\lambda}|_{\lambda=1} = \alpha = F^{-1}\dd F.$
Assume  $g \in G$ satisfies $g f_\lambda = f_\lambda$ for any fixed $\lambda \in S^1$.  Then  $g F_\lambda = F_\lambda k $ for some $ k : {M} \rightarrow K.$  A differentiation of this equation yields $g F_\lambda \alpha_\lambda = F_\lambda \alpha_\lambda k + F_\lambda\dd k ,$ whence $k \alpha_\lambda =   \alpha_\lambda k + \dd k.$ This equation is equivalent to $\dd k = [k, \alpha_\mathfrak{k}]$ and
 { $0 = [k, \alpha'_\mathfrak{p}], 0 = [k, \alpha''_\mathfrak{p}]$.}
  Note that this system of equations is independent of $\lambda$, i.e.
 $\dd k = [k, \alpha_\mu]$ holds for all $\mu \in S^1$.
Therefore, for the case $\mu = 1$ we obtain for $F = F_{\mu = 1}$ and $\alpha = \alpha_{\mu}|_{\mu=1}$ the equation
$\dd( FkF^{-1}) = F \alpha k F^{-1} + F \dd k F^{-1} - F k F^{-1} F \alpha F^{-1} =  F( \dd k - [k, \alpha]) F^{-1} = 0$.
Therefore, $FkF^{-1} = g_1 \in G$ is independent of $z$, and $g_1 f = f$, where we recall $f = f_{\lambda}|_{\lambda =1}$.  {Since $f$ is full, {$g_1 \in Center(G) $}. But then
$k=F^{-1}g_1F=g_1$ is contained in $Center(G) $. As a consequence, $g F_\lambda = F_\lambda k = kF_\lambda$ and $g = k \in Center(G) $ follows.} Hence $f_\lambda$ is full.
\end{proof}
%%%%%%%%%%%%%%%%%%%%%%

\subsection{Totally symmetric harmonic maps}

%%%%%%%%%%%%%%%%%%%%%%

\begin{definition} Let $M$ be a connected Riemann surface.
Let $f : M \rightarrow G/K$ be a harmonic map  and $f_\lambda$ its associated family in the sense of Section 2.1.
Then $f$ is called {\bf totally symmetric} if each member
$\hat{f}_\lambda : \hat{M} \rightarrow G/K$ of its associated family descends to a harmonic map
$f_\lambda : M \rightarrow G/K$.
\end{definition}
From the definition  we derive now
\begin{lemma} \label{chartotsym}
Let $f : M \rightarrow G/K$ be a full  harmonic map and $\hat{f} : \hat{M} \rightarrow G/K$
its natural lift to the universal cover $\hat{M}$ of $M$. Then
the following statements are equivalent
\begin{enumerate}
\item $f$ is totally symmetric,
\item Each member $\hat{f}_{\lambda}$ of the associated family
 of $\hat{f}$ satisfies
\begin{equation*}
\gamma^*\hat{f}_{\lambda} = \hat{f}_{\lambda} \hspace{2mm} \mbox{for all} \hspace{2mm} \gamma \in \pi_1(M),
\end{equation*} \item The extended frame $F_\lambda : \hat{M} \rightarrow \Lambda G_\sigma$ satisfies
\begin{equation}\label{eq-totsym}
\gamma^*F_\lambda =  F_\lambda k_\gamma \hspace{2mm} \mbox{for all} \hspace{2mm} \gamma \in \pi_1(M) \hspace{2mm} \mbox{and with } \hspace{2mm}  k_\gamma \in K.
\end{equation}

\end{enumerate}
\end{lemma}
%%%%%%%%%%%%%%%%%%%%%%%

\subsection{Totally symmetric harmonic maps defined on compact Riemann surfaces}

%%%%%%%%%%%%%%%%%%%%%%%

The main goal of this subsection is the
\begin{proof} of Theorem \ref{totally symmetric}:
We first show that $f_ {{U}, \lambda}$ is totally symmetric if  $f$ is totally symmetric. The converse is obtained by reversing all steps.

We consider the extended family $F_\lambda : \hat{M} \rightarrow \Lambda G_\sigma$ of $f$ and the
{unique} Iwasawa decomposition
\begin{equation}\label{eq-*}
F_\lambda = F_{{U},\lambda} V_+
\end{equation}
as in the proof of Theorem \ref{thm-noncompact}.
 Applying \eqref{eq-totsym} in (3) of Lemma \ref{chartotsym} we derive
\[\gamma^*F_{ {U},\lambda} \gamma^*V_+ =   \gamma^*F_\lambda =
F_\lambda k_\gamma  =  F_{ {U},\lambda} V_+ k_\gamma.\]
 Comparing the first expression to the last expression we obtain
 $ {  F_{ {U},\lambda}^{-1}   \gamma^*F_{ {U},\lambda} =  }  V_+ k_\gamma (\gamma^*V_+)^{-1}=\hat{k}_\gamma \in K$. This implies
\begin{equation}\label{eq-**} \hspace{3mm} \gamma^*F_{ {U},\lambda} = F_{ {U},\lambda}  \hat{k}_\gamma.
\end{equation}
From this it follows immediately that each member $\hat{f}_{\hat{U},\lambda}$
of the associated family of    ${\hat{f}_ {{U}}}: \hat{M} \rightarrow  {U}/ ({U} \cap  {K}^{\mathbb{C}})$ as defined in Theorem \ref{thm-noncompact},  descends to a harmonic map    $f_{\ {U}, \lambda}: M \rightarrow   {U}/ ({U} \cap  {K}^{\mathbb{C}}).$   In particular,  $f_ {{U}, \lambda}$ is totally symmetric.
\end{proof}

%%%%%%%%%%%%%%%%%%%%%%%%%%%%%%%%%%%%%

\section{ Applications to finite uniton type harmonic maps}

%%%%%%%%%%%%%%%%%%%%%%%%%%%%%%%%%%%%

The notion of {\bf finite uniton}  type of a harmonic map  was coined by Uhlenbeck in \cite{Uh}.
For this definition she required special properties of  ``extended solutions'', objects used extensively in that paper. Moreover, her definition was in the context of maps defined on $S^2$ or simply-connected subsets of $S^2$. In \cite{BuGu} the definition was extended (also by restrictions  to extended solutions) to harmonic maps from arbitrary Riemann surfaces $M$ to (compact inner) symmetric spaces.

In our work we primarily use  extended frames (not extended solutions).
And in view of  Theorem 3.11 of \cite{DoWa12} this also makes sense for
$M = S^2$. Therefore we prefer to give the definition of ``finite uniton type'' in terms of extended frames.

For this purpose we introduce the notion of  an {\it algebraic loop} as meaning  that the  Fourier expansion in $\lambda$ is a Laurent polynomial, i.e. it contains only finitely many powers  of $\lambda$. Such loops
will be denoted by the subscript $``alg"$, like  $\Lambda_{alg} G_{\sigma},\ \Lambda_{alg} G^{\mathbb{C}}_{\sigma}$,  { and $\Omega_{alg} G_{\sigma},$  where $\Omega G_\sigma $ consists of all loops in
$ \Lambda G_{\sigma}$ satisfying $\gamma(\lambda)|_{\lambda = 1} = \e$, the ``based loops".}
 We define
 \[ \Omega^k_{alg} G_{\sigma}:=\{\gamma\in\Omega_{alg} G_{\sigma}|\gamma=\sum_{|j|\leq k}\lambda^jT_j \}.
\]
 {We refer to page 545-547 of \cite{BuGu} for more details.}

Now we define the notion of {\it finite uniton type}.
\begin{definition}\label{def-uni}  Let $M$ be a Riemann surface, compact or non-compact. A harmonic map $f : M\rightarrow G/K$, where $G/K$ is a compact or non-compact symmetric space,  is said to be of {\it finite uniton type} if some extended frame $F$ of $f$,  defined on the universal cover
$\hat{M}$ of $M$ and satisfying $F(z_0,\bar{z}_0,\lambda)=e$ for some base point $z_0\in M$, has the following two properties:
\begin{enumerate}[$(U1)$]
\item $F (z,\bar{z},\lambda)$ descends to a map from $M$ to $(\Lambda G^{\C}_{\sigma})/K$, i.e. the map $F(z,\bar{z},\lambda):M\rightarrow (\Lambda G^{\C}_{\sigma})/K$ is well defined on $M$ (up to two singularities in the case of $M = S^2$) for all $\lambda \in S^1$.
\item  $F(z,\bar{z}, \lambda)$ is a Laurent polynomial in $\lambda$.
\end{enumerate}
\end{definition}
\vspace{2mm}

Hence $f$ is of finite uniton type if and only if there exists an extended frame $F$ for $f$ which has a trivial monodromy representation and is a Laurent polynomial in $\lambda$. In particular, in this case  $F(M) \subset \Omega^k_{alg} G_{\sigma}$ for some $k$.
It is also easy to verify that $f$ has finite uniton type if and only if $F_-$,  obtained from $F$ by the Birkhoff decomposition
$F = F_- F_+$ with  $F_- = I + \mathcal{O} (\lambda^{-1})$, descends to a map $F_- : M \rightarrow \Lambda^-_*G^\C_\sigma$ which is a Laurent polynomial.

\begin{lemma}
If ${M}$ is a simply connected Riemann surface and  $f : M \rightarrow G/K$ a harmonic map
into the symmetric space $G/K$,
then $f$ is of finite uniton type if and only if $f$ has an extended frame which is a Laurent polynomial in $\lambda$.
\end{lemma}

\begin{definition}
 Let $f:M \rightarrow G$  be a harmonic map of finite uniton type from $M$ into $G/K$ and $F$ an extended frame for $f$ satisfying $(U1)$ and $(U2)$.
We say that $f$ has {\it finite uniton number k} if
$$ F(M)\subset \Omega^k_{alg} G_{\sigma},\
\hbox{ and } F(M)\nsubseteq \Omega^{k-1}_{alg} G_{\sigma}.$$
 In this case we write  $r(f)=k$.
\end{definition}

In \cite{DoWa2} the notion of finite uniton type is investigated in much more detail. In particular, it is shown there that the definition above coincides with the definition of \cite{BuGu} and \cite{Uh}. Combining the above definition with Theorem \ref{thm-noncompact} we obtain:

\begin{theorem}\label{thm-fut}
 We retain the notation of Theorem \ref{thm-noncompact}.
\begin{enumerate}
\item Let $f: {M}\rightarrow  {G}/ {K}$  be a harmonic map and $f_{ {U}}$ the associated harmonic map into the compact symmetric space $ {U}/( {U}\cap  {K}^{\mathbb{C}})$  as in Theorem \ref{thm-noncompact}.
Then $f$ is of finite uniton type if and only if $f_{ {U}}$ is of finite uniton type.

Moreover, in this case we have $r(f) = r(f_{ {U}})$.
\item If $  M=S^2$, then $f$ is of finite uniton type.
\end{enumerate}
\end{theorem}
\begin{proof}
 (1) Let $F$ and $F_{ {U}}$ be the (local) extended frame of $f$ and $f_{ {U}}$ respectively and assume that $f$ or $f_{ {U}}$ is of finite uniton type and $F$ and $F_{ {U}}$ the corresponding frame respectively.
 By \eqref{eq-noncompact},  the Fourier expansion of $F$ contains only finitely many  powers of $\lambda^{-1}$ if and only if $F_{ {U}}$ has only finitely many powers of $\lambda^{-1}$ in its Fourier expansion. Using the reality of $F$ and $F_{ {U}}$, we conclude that $F$ is a Laurent polynomial of $\lambda$ if and only if  $F_{ {U}}$ is a Laurent polynomial of $\lambda$.  The equality of $r(f)$ and $r(f_{ {U}})$ follows,
 since the extremal (negative, whence positive) power of $\lambda$ occurring in $F$ and $F_{ {U}}$ are the same. Finally, by  \eqref{eq-noncompact}
it is easy to check that $F\mod K$ is defined on $M$ if and only if
$F_{ {U}}\mod  {U} \cap  {K}^\mathbb{C}$ is defined on $M$.

(2). In the case of $ {M}=S^2$, by  {what we have just shown in the proof of $(1)$, it suffices to verify }that
$f_{ {U}}$ is of finite uniton type.
 To show this, we  first claim that $\mathbb{F}(z,\bar z)=F_{ {U}}(z,\bar z, -1)F_{ {U}}(z,\bar z, 1)^{-1}$ is a harmonic map into $ {U}$ having $\Phi(z,\bar z,\lambda)=F_{ {U}}(z,\bar z, \lambda)F_{ {U}}(z,\bar z, 1)^{-1}$
 as its extended solution. Here $F_{ {U}}(z,\bar z, \lambda)$ is the extended frame of $f_{  U}$ with the
 Maurer-Cartan form
  $\alpha_{ {U}}=\lambda^{-1} \alpha'_{\mathfrak{p}}+
\alpha_{\mathfrak{k}}+
\lambda\alpha''_{\mathfrak{p}}$ and $F_{ {U}}(z_0,\bar z_0,\lambda)=\e$. Straightforward computations show
\[\mathbb{A}=\frac{1}{2}\mathbb{F}^{-1}\dd\mathbb{F}=-F_{ {U}}(z,\bar z, 1)(\alpha'_{\mathfrak{p}}+\alpha''_{\mathfrak{p}})F_{ {U}}(z,\bar z, 1)^{-1}=\mathbb{A}^{(1,0)}+\mathbb{A}^{(0,1)},\]
and
 \[\Phi(z,\bar z,\lambda)^{-1}\dd \Phi(z,\bar z,\lambda)=(1-\lambda^{-1})\mathbb{A}^{(1,0)}+(1-\lambda)\mathbb{A}^{(1,0)}. \]
 The claim now follows from Theorem 2.1 of \cite{Uh}. Moreover, since $\Phi(z_0,\bar z_0,\lambda)=\e$ by Theorem 11.5 of \cite{Uh},  the extended solution $\Phi(z,\bar z,\lambda)$ is a Laurent polynomial of $\lambda$. Hence $F_{ {U}}(z,\bar z, \lambda)$ is a Laurent polynomial of $\lambda$ and, as a consequence, $f_{ {U}}$ is of finite uniton type.
\end{proof}

We want to prove a result which shows that certain types of harmonic maps are of finite uniton type. By $(U1)$ above and in view of $(3)$ of Lemma \ref{chartotsym} it is clear that a finite uniton map always is totally symmetric. Thus one needs to look for a condition on a harmonic map which ensures the existence of an extended frame which is a Laurent polynomial in $\lambda$.
Here Theorem 11.5 of Uhlenbeck  is of importance, see \cite{Uh}.
Note, this theorem holds for all compact Riemann surfaces by \cite{Segal}. We thus claim
\begin{theorem} \label{M compact} {Let $M$ be a compact Riemann surface.}
 \begin{enumerate}
 \item {Any totally symmetric harmonic map from $M$ to any compact inner symmetric space is of finite uniton type.}
\item Let $f:M \rightarrow  {G}/ {K}$ be a totally symmetric harmonic map into some non-compact symmetric space.   {Then $f_U : M \rightarrow  U/ {U\cap K^{\C}}$ is of finite uniton type.}
\end{enumerate}
\end{theorem}
\begin{proof}
(1)  Uhlenbeck has proven (Theorem 11.5 of \cite{Uh})  that for harmonic maps from $S^2$ to
any  unitary group $U(n)$, one can always find an extended solution which is a Laurent polynomial in $\lambda$.
By Lemma 2.15 of \cite{Segal},  for all harmonic maps from any compact Riemann  surface into any inner compact symmetric space, one can always find an extended solution which is a Laurent polynomial in $\lambda$.   {This result also holds for extended frames: For the case $M = S^2$ see e.g. Remark 3.10 and Theorem 3.11 of \cite{DoWa12}. For all other compact Riemann surfaces $M$ a careful inspection of the arguments of \cite{Segal} shows that its result actually holds for all totally symmetric harmonic maps from any  compact Riemann surface to any inner compact symmetric space (also see Theorem 1.2 of \cite{BuGu}).}  Therefore, for any totally symmetric harmonic map $\check{f}: M \rightarrow \check{U} / \check{H}$  into the compact inner symmetric space $ \check{U} / \check{H}$  with defining symmetry $\check{\sigma}$
  there exists some $A(\lambda) \in \Lambda  {\check{U}} _{\check{\sigma}}$ such that
   {$A(\lambda) F_{{\check{U}} }$ is a Laurent polynomial. In poarticular, since
 $F_{\check{U}}|_{z=z_0}=e$, we have that $A(\lambda)$ is a Laurent polynomial in $\lambda$. As a consequence,
 $F_{\check{U}}$ is also a Laurent polynomial in $\lambda$, i.e., $\check{f}$ is of finite uniton type.}

(2).
Since $f$ is totally symmetric, the dual harmonic map $f_ { {U}}: M \rightarrow  {U}/ ( {U} \cap  {K}^{\mathbb{C}})$   is also totally symmetric by Theorem \ref{totally symmetric}. So by (1), $f_{U}$ is of finite uniton type. Then by (1) of Theorem \ref{thm-fut}, $f$ is of finite uniton type.
\end{proof}

\begin{remark}
 { We conjecture that the proof of \cite{Segal} actually shows that also  each $f$ as in $(2)$ is of finite uniton type.}
\end{remark}

%%%%%%%%%%%%%%%%%%%%%%%%%%%

\section{Examples}

%%%%%%%%%%%%%%%%%%%%%%%%%%%

We end this paper with an example of the conformal Gauss map $f$ of a Willmore surface and the compact dual harmonic map $f_U$ of $f$ (See  \cite{Wang-iso} for more details).
\begin{example} \cite{Wang-iso}
We consider  $SO^+(1,5)$ as the connected component of $SO(1,5)=\{A\in Mat(6,\R)\ | \ A^tI_{1,5}A=I_{1,5}, |A|=1\}$ containing $I$. Here $I_{1,5}=\hbox{diag}(-1,1,1,1,1,1)$. Let $f_2$ and $f_4$ be meromorphic functions on $M$ with $f_2(z_0)=f_4(z_0)=0$ at some base point $z_0\in M$, where $M$ is
a compact Riemann surface of genus $\geq 0$. Let
\begin{equation}\eta=\lambda^{-1}\left(
                      \begin{array}{cc}
                        0 &  {B}_1 \\
                        - {B}_1^tI_{1,3} & 0 \\
                      \end{array}
                    \right)dz,\ \hbox{ with }\   {B}_1=\frac{1}{2}\left(
                     \begin{array}{cccc}
                      -if_2' &  f_2'  \\
                      if_2'&  -f_2'  \\
                      f_4' & if_4'   \\
                      if_4' & -f_4'  \\
                     \end{array}
                   \right),
                   \end{equation}
                   be the normalized  potential of a harmonic map $f(z,\bar z, \lambda): M\rightarrow SO^+(1,5)/SO^+(1,3)\times SO(2)$. Then the extended frame $F(z,\bar z, \lambda)=(\phi_1,\phi_2,\phi_3,\phi_4,\phi_5,\phi_6)$ of $f(z,\bar z, \lambda)$ is
                   \begin{equation}\left(
                       \begin{array}{cccccc}
                         1+\frac{|f_2|^2}{2} &\frac{|f_2|^2}{2}  & \frac{-i(\bar{f}_2f_4-f_2\bar f_4)}{2\sqrt{d_3}} & \frac{\bar{f}_2f_4+f_2\bar f_4}{2\sqrt{d_3}} &
                         \frac{-i(\lambda^{-1}f_2-\lambda\bar{f}_2)}{2\sqrt{d_3}}&   \frac{\lambda^{-1}f_2+\lambda\bar{f}_2}{2\sqrt{d_3}} \\
                         -\frac{|f_2|^2}{2} & 1-\frac{|f_2|^2}{2} & \frac{i(\bar{f}_2f_4-f_2\bar f_4)}{2\sqrt{d_3}} & -\frac{\bar{f}_2f_4+f_2\bar f_4}{2\sqrt{d_3}}  &\frac{i(\lambda^{-1}f_2-\lambda\bar{f}_2)}{2\sqrt{d_3}}&   -\frac{\lambda^{-1}f_2+\lambda\bar{f}_2}{2\sqrt{d_3}} \\
                         0 & 0&  \frac{1}{\sqrt{d_3}} & 0& \frac{\lambda^{-1}f_4+\lambda\bar{f}_4}{2\sqrt{d_3}} & \frac{i(\lambda^{-1}f_4-\lambda\bar{f}_4)}{2\sqrt{d_3}} \\
                         0 & 0 & 0 &  \frac{1}{\sqrt{d_3}}& \frac{i(\lambda^{-1}f_4-\lambda\bar{f}_4)}{2\sqrt{d_3}}&  -\frac{\lambda^{-1}f_4+\lambda\bar{f}_4}{2\sqrt{d_3}}  \\                         \frac{-i(\lambda^{-1} f_2-\lambda\bar f_2)}{2} & \frac{-i(\lambda^{-1}f_2-\lambda\bar f_2)}{2}  &-\frac{\lambda^{-1}f_4+\lambda\bar{f}_4}{2\sqrt{d_3}} & \frac{-i(\lambda^{-1}f_4-\lambda\bar{f}_4)}{2\sqrt{d_3}} &\frac{1}{ \sqrt{d_3}}&0 \\
                         \frac{\lambda^{-1}f_2+\lambda\bar f_2}{2}  & \frac{\lambda^{-1}f_2+\lambda\bar f_2}{2}    & \frac{-i(\lambda^{-1}f_4-\lambda\bar{f}_4)}{2\sqrt{d_3}} &\frac{\lambda^{-1}f_4+\lambda\bar{f}_4}{2\sqrt{d_3}} &0&\frac{1}{ \sqrt{d_3}} \\
                       \end{array}
                     \right).
                   \end{equation}
                    Here $d_1=1+|f_2|^2+|f_4|^2$, $d_2=1+|f_2|^2$ and $d_3=1+|f_4|^2$.                    Note that $f_{\lambda}$ is the conformal Gauss map of the minimal surface $x_{\lambda}$ in $\mathbb{R}^4$:
\begin{equation}\label{eq-min-iso2}x_{\lambda}=\left(
                                             \begin{array}{cccc}
                        -\frac{if_2'}{f_4'} +\frac{i\overline{f_2'}}{\overline{f_4'}} \\
                        -\frac{f_2'}{f_4'}-\frac{\overline{f_2'}}{\overline{f_4'}} \\
                     -i(\lambda^{-1}f_2-\lambda\bar{f}_2)+\frac{i\lambda^{-1}f_2'f_4}{f_4'}-\frac{i\lambda\overline{f_2'}\bar{f}_4}{\overline{f_4'}} \\
                     (\lambda^{-1}f_2+\lambda\bar{f}_2)-\frac{\lambda^{-1}f_2'f_4}{f_4'}-\frac{\lambda\overline{f_2'}\bar{f}_4}{\overline{f_4'}} \\
                                                              \end{array}
                                           \right). \end{equation}
The extended frame  $ F_U(z,\bar z, \lambda)$  of the dual harmonic map $f_U(z,\bar z, \lambda)$ into the dual compact symmetric space $SO(6)/SO(4)\times SO(2)$ is
\begin{equation}\label{eq-compact-dual}\left(
                       \begin{array}{cccccc}
                         \frac{1}{\sqrt{d_2}} &0 & \frac{\bar{f}_2f_4+f_2\bar f_4}{2\sqrt{d_1d_2}} & \frac{i(\bar{f}_2f_4-f_2\bar f_4)}{2\sqrt{d_1d_2}}&   -\frac{\lambda^{-1}f_2+\lambda\bar f_2}{2\sqrt{d_1}} & -\frac{i(\lambda^{-1}f_2-\lambda\bar f_2)}{2\sqrt{d_1}}   \\
                         0 & \frac{1}{\sqrt{d_2}}& \frac{i(\bar{f}_2f_4-f_2\bar f_4)}{2\sqrt{d_1d_2}} & \frac{-\bar{f}_2f_4-f_2\bar f_4}{2\sqrt{d_1d_2}} &
                          \frac{i(\lambda^{-1}f_2-\lambda\bar f_2)}{2\sqrt{d_1}} & -\frac{\lambda^{-1}f_2+\lambda\bar f_2}{2\sqrt{d_1}}  \\
                         0 & 0&  \frac{\sqrt{d_2}}{\sqrt{d_1}} & 0 &  \frac{\lambda^{-1}f_4+\lambda\bar f_4}{2\sqrt{d_1}}    &   \frac{i(\lambda^{-1}f_4-\lambda\bar f_4)}{2\sqrt{d_1}}  \\
                         0 & 0 & 0 &   \frac{\sqrt{d_2}}{\sqrt{d_1}}  &
                          \frac{i(\lambda^{-1}f_4-\lambda\bar f_4)}{2\sqrt{d_1}} & -\frac{\lambda^{-1}f_4+\lambda\bar f_4}{2\sqrt{d_1}}  \\
                         \frac{\lambda^{-1}f_2+\lambda\bar f_2}{2\sqrt{d_2}}  &\frac{-i(\lambda^{-1}f_2-\lambda\bar f_2)}{2\sqrt{d_2}}
                         &-\frac{\lambda^{-1}f_4+\lambda\bar{f}_4}{2\sqrt{d_1d_2}} & \frac{-i(\lambda^{-1}f_4-\lambda\bar{f}_4)}{2\sqrt{d_1d_2}} & \frac{1}{\sqrt{d_1}} &0 \\
                         \frac{i(\lambda^{-1}f_2-\lambda\bar f_2)}{2\sqrt{d_2}}  &  \frac{\lambda^{-1}f_2+\lambda\bar f_2}{2\sqrt{d_2}}  & \frac{-i(\lambda^{-1}f_4-\lambda\bar{f}_4)}{2\sqrt{d_1d_2}} &\frac{\lambda^{-1}f_4+\lambda\bar{f}_4}{2\sqrt{d_1d_2}}& 0 &\frac{1}{\sqrt{d_1}} \\
                       \end{array}
                     \right).
                   \end{equation}
Note that all  these examples are totally symmetric harmonic maps of finite uniton type on $M$.

\end{example}

\begin{example}  We consider the influence of change of initial conditions explicitly for the above example.
Set $\chi=F(z,\bar z, \lambda)|_{z=z_1}$. Then $\hat{F}(z,\bar z, \lambda)=\chi^{-1}F(z,\bar z, \lambda)$ is  the extended frame  of  $\hat f(z,\bar z, \lambda)=\chi^{-1}f(z,\bar z, \lambda)$. One computes that the normalized potential of   $\hat f(z,\bar z, \lambda)$ is
\begin{equation}\label{eq-hat-eta}
\hat{\eta}=\lambda^{-1} \left(
                      \begin{array}{cc}
                        0 & \hat{B}_1 \\
                        -\hat{B}_1^tI_{1,3} & 0 \\
                      \end{array}
                    \right)\dd z,  \hbox{ with }
                \hat{B}_1= \frac{1}{2}\left(
                     \begin{array}{cccc}
                      -i\hat{f}_2' &  \hat{f}_2'   \\
                      i\hat{f}_2'&  -\hat{f}_2' \\
                      \hat{f}_4' & i\hat{f}_4'  \\
                      i\hat{f}_4' & -\hat{f}_4'  \\
                     \end{array}
                   \right).
                    \end{equation}
                    Here
    \begin{equation}
\hat f_2= \frac{ \sqrt{1+|{f}_4(z_1)|^2}(f_2-f_2(z_1))}{1+\bar{f}_4(z_1)f_4},\ \hat{f}_4=\frac{f_4-f_4(z_1)}{1+\bar{f}_4(z_1)f_4}.
\end{equation}
Substituting $\hat f_2$ and $\hat f_4$ into \eqref{eq-compact-dual}, we obtain the  extended frame  $\hat{F}_U(z,\bar z, \lambda)$  of the dual harmonic map $\hat{f}_U(z,\bar z, \lambda)$   of  $\hat f(z,\bar z, \lambda)$. One can check that in general  $\hat{f}_U(z,\bar z, \lambda)$  is not congruent to  $f_U(z,\bar z, \lambda)$  up to $\Lambda SO(6)_{\sigma}$.
In fact, we have
\[ |\dd f|^2=\frac{2|f_4'|^2|\dd z|^2}{(1+|f_4|^2)^2},\]
\[|\dd\tilde f|^2=\frac{2(|f_2'|^2d_1+|f_4'd_2 -f_2'f_4\bar{f}_2|^2)|\dd z|^2}{d_1^2d_2}=\frac{2(|f_2'|^2+|f_4'|^2+|f_2f_4'-f_4f_2'|^2|)\dd z|^2}{(1+|f_2|^2+| f_4|^2)^2}.\]
So in general
$|\dd f|^2=|\dd \hat f|^2$, while  $|\dd\tilde{\hat{f}}|^2\neq|\dd\tilde f|^2$ in general.
For example, set $f_2=z^2$, $f_4=z$, $z_0=0$ and $z_1=1$. Then
\[|\dd\tilde{f}|^2=\frac{2(r^4+4r^2+1)|\dd z|^2}{(r^4+r^2+1)^2}\neq|\dd\tilde{\hat{f}}|^2=\frac{2(r^4+4r^2+1+z^2+\bar{z}^2)|\dd z|^2}{(r^4+r^2+1-z^2-\bar{z}^2)^2}.\]
In particular, $\tilde{f}$ is not congruent to $\tilde{\hat{f}}$ up to $SO(6)$.
\end{example}

{\footnotesize
\def\refname{Reference}

}

\vspace{2mm}

{\footnotesize \begin{multicols}{2}

Josef F. Dorfmeister

Fakult\" at f\" ur Mathematik,

TU-M\" unchen, Boltzmann str. 3,

D-85747, Garching, Germany

{\em E-mail address}: dorfm@ma.tum.de\\

Peng Wang

School of Mathematics and Statistics, FJKLMAA,

Fujian Normal University, Qishan Campus,

Fuzhou 350117, P. R. China

{\em E-mail address}: {pengwang@fjnu.edu.cn}

\end{multicols}}


\begin{thebibliography}{11}



\bibitem {Aom} Aomoto, K.,{\em On some double coset decompositions of complex semi-simple Lie
groups,} J. Math. Soc. Japan, Vol. 18 (1966), 1-44.
\bibitem {Ba-Do} Balan, V., Dorfmeister, J. {\em Birkhoff decompositions and Iwasawa decompositions for loop groups,}
Tohoku Math. J. Vol. 53 (2001) , No.4, 593-615.

\bibitem {B-R-S}  Brander, D., Rossman, W.,  Schmitt, N. {\em Holomorphic representation of constant mean curvature surfaces in Minkowski space: Consequences of non-compactness in loop group methods,} Adv. Math. Vol. 223, No.3, (2010), 949-986.
\bibitem{BuGu} Burstall,  F.E., Guest, M.A., {\em Harmonic two-spheres in compact symmetric spaces, revisited,} Math. Ann. 309 (1997), 541-572.

%\bibitem {Do} Dorfmeister, J. {\em Open Iwasawa cells and applications to surface theory},
%Variational Problems in Differential Geometry. Cambridge University Press, 2011, 56-67.


\bibitem{Do} Dorfmeister, J., {\em Open Iwasawa cells and applications to surface theory} ,
in: Variational Problems in Differential Geometry,
       London Math.Soc. LN 394, p. 56-67, Cambridge Univ. Press 2012.

%\bibitem{Dol1} Dorfmeister, J., {\em   Weighted l1-Grassmannians and Banach Manifolds of Solutions to the KP-%Equation and the KdV-Equation, }
  %        Math. N. 180 (1996), 43-73.

\bibitem{Do-base} Dorfmeister, F. J., {\em
How does the loop group method depend on the basepoint}, in preparation, see arXiv

\bibitem{DGS} Dorfmeister, J., Gradl, H., Szmigielski, S., {\em Systems of PDEs Obtained from Factorization in Loop Groups,} Acta Applicandae Mathematicae 53 (1998), 1- 58.


\bibitem{DPW} Dorfmeister, J., Pedit, F., Wu, H., {\em Weierstrass type representation of harmonic maps
into symmetric spaces,} Comm. Anal. Geom. 6 (1998), 633-668.

\bibitem{DoWa11} Dorfmeister, J., Wang, P.,  {\em Weierstrass-Kenmotsu representation of Willmore surfaces in spheres,} Nagoya Math. J. 244 (2021), 35-59.

\bibitem{DoWa12}  Dorfmeister, J., Wang, P.,    {\em Willmore surfaces in spheres: the DPW approach via the conformal Gauss map,}
Abh. Math. Semin. Univ. Hambg: 2019 ,89, No.1 ,77-103.

\bibitem{DoWa2} Dorfmeister, J.,  Wang, P., {\em Harmonic maps of finite uniton type into non-compact inner symmetric spaces},   arXiv:1305.2514v2.

 \bibitem {Ejiri1988} Ejiri, N. {\em  Willmore surfaces with a duality in $S^{n}(1)$,} Proc. London Math. Soc. (3), 57(2) (1988), 383-416.

 \bibitem{Helgason} Helgason S., {\em Differential geometry, Lie groups, and symmetric spaces,} volume 34 of Graduate Studies in Mathematics. AMS, Providence, RI, 2001. Corrected reprint of the 1978 original.

     \bibitem{Hochschild} Hochschild, G.  {\em The structure of Lie groups,} Holden-Day Inc., San Francisco, 1965.

 \bibitem{PS} Pressley A.N., Segal, G.B., {\em Loop Groups,} Oxford University Press, 1986.

\bibitem{Segal} Segal, G.B., {\em Loop groups and harmonic maps} in Advances in homotopy theory, Cortona, 1988, 153-164,
London Math. Soc. Lecture Note Ser., 139, Cambridge Univ. Press, Cambridge, 1989.

\bibitem{Uh}Uhlenbeck, K. {\em Harmonic maps into Lie groups (classical solutions of the chiral model),} J. Diff. Geom. 30 (1989), 1-50.

\bibitem{Wang-1} Wang, P., {\em Willmore surfaces in spheres via loop groups II: a coarse classification of Willmore two-spheres by potentials}, arXiv:1412.6737.

\bibitem{Wang-iso} Wang, P., {\em Willmore surfaces in spheres via loop groups IV: on totally isotropic Willmore two-spheres in $S^6$}, Chinese Ann. Math. Ser. B 42 (2021), no. 3, 383-408.

\bibitem{Wang-a} Wang, P., {\em Construction of Willmore  two-spheres via harmonic maps into   $SO^+(1,n+3)/SO^+(1,1)\times SO(n+2)$}, Willmore energy and Willmore conjecture, 85-117, Monogr. Res. Notes Math., CRC Press, 2018.

%\bibitem{Wu} Wu, H.Y. {\em A simple way for determining the normalized potentials for harmonic maps,} Ann. %Global Anal. Geom. 17 (1999), no. 2, 189-199.

\end{thebibliography}
\end{document}